\documentclass{article}
\usepackage{graphicx, amsmath, amsfonts, amsthm, amssymb, enumitem, svg, subcaption}
\usepackage[unicode = false,
    pdftoolbar = true,
    colorlinks = true,
    linkcolor = blue,
    citecolor = blue,
    filecolor = black,
    urlcolor = blue,
    breaklinks = true]{hyperref}
\usepackage{listings}
\usepackage{comment}
\usepackage{float}
\usepackage[affil-it]{authblk}
\usepackage{todonotes}
\usepackage[section]{placeins}
\usepackage{algpseudocodex}
\usepackage{hyperref}
\usepackage{tikz,pgfplots}

\pgfplotsset{
    compat=1.18,
    /pgf/declare function={
        ellipsex(\a,\b,\c,\x) = \a*cos(\x r)+\b*sin(\x r);
        ellipsey(\a,\b,\c,\x) = \b*cos(\x r)+\c*sin(\x r);
        A1_11 = 0.6035533905932738;
        A1_12 = -0.10355339059327376;
        A1_22 = 0.6035533905932738;
        A2_11 = 0.4743416490252569;
        A2_12 = -0.15811388300841897;
        A2_22 = 1.1067971810589328;
        B0_11 = 0.36223802824035045;
        B0_12 = -0.04601026222351254;
        B0_22 = 0.3622380282403505;
        B1_11 = 0.4277309429819104;
        B1_12 = 0.01948265251804743;
        B1_22 = 0.4277309429819104;
        B2_11 = 0.4369861898826181;
        B2_12 = -0.014816897114708388;
        B2_22 = 0.5555213668002852;
        B1alt_11 = 0.4643001008563161;
        B1alt_12 = -0.14807233483947824;
        B1alt_22 = 0.4643001008563161;
        B2alt_11 = 0.4673804287312365;
        B2alt_12 = -0.14382617849895876;
        B2alt_22 = 0.5906600103017725;
    },
}

\title{An algorithm to exactly compute minimal upper bounds in the Loewner order}

\author{Adam Humeniuk}
\affil{Department of Mathematics and Computing, Mount Royal University, Calgary, 
              Alberta, Canada \\ Humeniuk's research is funded by the Natural Sciences and Engineering Research Council (NSERC) of Canada, Discovery Grant DG2024-03883, DGECR2024-00522\\ 
              ORCID 0000-0001-8874-9312 \\
              Email: ahumeniuk@mtroyal.ca}

\author{Gabriel Jarry-Bolduc}
\affil{Department of Mathematics and Statistics, American University of Sharjah, 
              Sharjah, United Arab Emirates
              \\ORCID 0000-0002-1827-8508 \\
              Email: gabjarry@alumni.ubc.ca }

\author{Patrick Pascua}
\affil{Department of Mathematics and Computing, Mount Royal University, Calgary, 
              Alberta, Canada \\
              Email: patrick.pascua@mtroyal.ca}

\author{Nejaunie Williams}
\affil{Department of Mathematics and Computing, Mount Royal University, Calgary, 
              Alberta, Canada \\ Williams' research is funded by the NSERC Undergraduate Student Research Award (USRA), award number 104536.
              Email: nejaunie.williams@mtroyal.ca}

\makeatletter
\def\@maketitle{%
  \newpage
  \null
  \vskip 2em%
  \begin{center}%
  \let \footnote \thanks
    {\Large\bfseries \@title \par}%
    \vskip 1.5em%
    {\normalsize
      \lineskip .5em%
      \begin{tabular}[t]{c}%
        \@author
      \end{tabular}\par}%
    \vskip 1em%
    {\normalsize \@date}%
  \end{center}%
  \par
  \vskip 1.5em}
\makeatother

\newtheorem{master}{Master}[section]
\newtheorem{definition}[master]{Definition}
\newtheorem{theorem}[master]{Theorem}
\newtheorem{lemma}[master]{Lemma}

\newtheorem{proposition}[master]{Proposition}

\theoremstyle{definition}
\newtheorem{example}[master]{Example}
\newtheorem{remark}[master]{Remark}

\newtheoremstyle{algorithmstyle}
  {\topsep}
  {\topsep}
  {\normalfont}
  {}
  {}
  {}
  {0pt}
  {\noindent\textbf{#1\ \thmnumber{#2.}}\\}

\theoremstyle{algorithmstyle}
\newtheorem{algorithm}[master]{Algorithm}

\newtheoremstyle{pseudostyle} 
  {\topsep}
  {\topsep}
  {\normalfont}
  {}
  {}
  {}
  {0pt}
  {\noindent\textbf{#1\ \thmnumber{#2.}}\\}
  
\theoremstyle{pseudostyle}
\newtheorem{pseudocode}[master]{PseudoCode}

\DeclareMathOperator{\nullspace}{null}
\newcommand{\col}{\text{col}}
\newcommand{\spn}{\text{span}}

\newcommand{\vv}{\mathbf{v}}
\newcommand{\e}{\mathbf{e}}
\newcommand{\uu}{\mathbf{u}}
\newcommand{\x}{\mathbf{x}}
\newcommand{\y}{\mathbf{y}}
\newcommand{\sa}{\mathrm{h}} 
\newcommand{\zero}{\mathbf{0}}

\newcommand{\R}{\mathbb{R}}
\newcommand{\C}{\mathbb{C}}
\newcommand{\F}{\mathbb{F}}
\newcommand{\ph}{\phantom{-}}
\newcommand{\bbm}{\begin{bmatrix}}
\newcommand{\ebm}{\end{bmatrix}}
\newcommand{\Id}{\mathtt{I}}
\newcommand{\commonscale}{0.9}

\begin{document}
\maketitle


\begin{abstract}
The \emph{Loewner order} on \emph{Hermitian matrices} is a partial order that compares matrices in terms of  \emph{positive semidefiniteness}. The Loewner order plays a key role in many fields such as optimization, numerical linear algebra, control theory, operator theory, and quantum information. A fundamental difficulty is that two or more Hermitian matrices do not necessarily have a unique minimal upper bound (or maximal lower bound). In this paper, we propose an iterative method to exactly compute a minimal upper bound for any finite collection of $n\times n$ Hermitian matrices. It is shown that the algorithm terminates in at most $n$ iterations. The exactitude of the algorithm is proved using standard results from finite-dimensional linear algebra. A self-contained proof of an algebraic characterization of minimality originally explored by Stott is provided. We illustrate the algorithm in examples and also provide an implementation of the algorithm in Python.
\end{abstract}

\section{Introduction} \label{sec:introduction}

The notion of \emph{Hermitian matrices}  is a fundamental topic in mathematics and physics named after French mathematician Charles Hermite. One of the  most important properties of Hermitian matrices is that all eigenvalues  are real numbers.  As a consequence, the concept of \emph{positive (semi)definiteness} is well-defined for Hermitian matrices.  Their structure is exploited in different fields of applied mathematics such as \emph{semidefinite programming} \cite{vandenberghe1996semidefinite}, \emph{Cholesky factorization} \cite{GolubVanLoan2013,PAN1993103}, the \emph{conjugate gradient method} \cite{HesteneStiefel1952}, \emph{preconditioning} \cite{AshbyManteuffelSaylor1989},  and \emph{covariance matrices} \cite[Chapter 3]{Anderson2003} to name a few. 
To compare Hermitian matrices,  the \emph{Loewner order} may be used. A matrix $B$ is said to be greater than or equal to  a matrix $A$ if and only if  $B-A$ is \emph{positive semidefinite}. The Loewner order is a \emph{partial order} but not a \emph{total order} (see \cite{BAKSALARY1990171,BAKSALARY1989,BAKSALARY1990} for more details on matrix partial orderings).  This means that some pairs of matrices are not comparable. It was introduced in 1934 by Charles Loewner in  \cite{Loewner1934}. The paper investigates \emph{matrix monotone functions}--a type of matrix functions preserving  the Loewner order. We may think of matrix monotone functions and the Loewner order as a generalization of increasing  real-valued functions to the matrix  setting. Likewise, the corresponding notion of a \emph{matrix convex function} is equally important. For thorough treatments of these topics, we refer the reader to \cite[Chapter V]{bhatia2013matrix} and \cite{simon2019loewner}.

Note that not all authors refer to the Loewner order by Loewner's name. For example, the well-known optimization text \cite{boyd2004convex} simply refers to the \emph{generalized inequality} associated to the positive semidefinite cone, or even simply ``the usual matrix inequality".

Basic properties of the Loewner order are described in  \cite{bhatia2013matrix,bhimasankaram2010matrix}. The Loewner order is a partial order that is not total, and so one might hope that this order has lattice properties such as the existence of unique least upper bounds (or greatest lower bounds). Kadison \cite{kadison1951} proved that this is as far from being true as possible. Two matrices have a unique least upper bound (or greatest upper bound) if and only if they are comparable, which Kadison called the ``anti-lattice" property. (Moreover, Kadison proves that this is true even for bounded operators on any infinite dimensional Hilbert space. In fact, the Loewner order is closely connected to the study of operator algebras, and Sherman \cite{sherman1951order} showed that the lattice properties always fail in any noncommutative operator algebra.)

While finding a \emph{unique} minimum upper bound is usually impossible, there always exist \emph{minimal} upper bounds, which have the property that there is no strictly smaller upper bound in the Loewner order. This might be viewed as a matricial analogue of Pareto optimality, which appears frequently in vector optimization (cf. \cite[Section 4.7]{boyd2004convex}). In \cite[Theorem 1.15]{stott2017minimal}, Stott provides a geometric, kernel-based certificate for minimality in the Loewner order. Specifically, if a real symmetric matrix $B$ is an upper bound for the set of real $n\times n$ symmetric matrices $\{A_1,\ldots,A_k\}$, then $B$ is a minimal upper bound if and only if the associated null spaces satisfy
\[
\sum_{i=1}^k \nullspace(B-A_i) = \R^n.
\] 
That is, $B$ must agree with some $A_i$ on enough vectors to span the whole space. This also holds for complex Hermitian matrices on the space $\C^n$, and we provide a self-contained proof below in Proposition \ref{minimal condition}.

The main goal of this paper is to extend the work contained in \cite{stott2017minimal} by  developing an iterative method  that produces a minimal upper bound for a finite set of $k$ Hermitian matrices of any finite size $n\times n$. It provides a simple iterative method that outputs a minimal upper bound for a given non-empty finite set of  Hermitian  matrices using any initial upper bound matrix as an input, and the minimal upper bound obtained is guaranteed to be less than the initial upper bound matrix in the Loewner order. It is shown that the algorithm returns an exact minimal upper bound matrix with at most $n$ iterations.  A pseudo-code of the algorithm  is introduced (see Algorithm \ref{main_algorithm_code}) and the algorithm is implemented in Python (\href{https://github.com/patrickp78-byte/Minimal-Upper-Bound-Algorithm-in-Loewner-Ordering}{GitHub Repository}). The main contribution of this paper is to provide a simple computational method that returns an exact minimal upper bound for any (non-empty) finite set of finite-dimensional Hermitian matrices. Note that the algorithm can also be used to compute maximal lower bounds for a finite set of Hermitian matrices, by simply negating all matrices involved. Possible applications where the algorithm may be valuable  include  finding a good starting point  in semidefinite programming optimization problems, or identifying the \emph{outer Loewner-John ellipsoid} (see  \cite{henk2012lowner}). More potential applications are discussed in the conclusion. 

The remainder of this paper is organized as follows. In Section \ref{sec:preliminaries}, the notation and fundamental results necessary to understand  the main results of this paper are introduced.  In Section \ref{sec:main}, we give a self-contained proof of the minimality condition and then introduce a pseudo-code of an algorithm to find an exact minimal upper bound and prove its veracity. Then, we demonstrate the algorithm in exact examples. In Section \ref{sec:conclusion}, the main contributions of this paper are summarized, many possible applications in which the iterative method could be employed are proposed, and future research directions are suggested.

\section{Preliminaries} \label{sec:preliminaries}

Throughout this paper, the notation $\mathbb{F}$ denotes either the set of complex numbers  $\mathbb{C}$, or the set of real numbers  $\mathbb{R}$. A matrix is denoted with an upper case letter such as $A$ and a vector is denoted with a lower case bold face letter such as $\vv$. The transpose of a matrix $A$ is denoted by $A^\top$. We denote the Hermitian conjugate  of a matrix $A$ by $A^\ast$, so that $A^\ast = A^\top$ when $A$ is a real-matrix. For vectors $\x,\y\in \F^n$, we denote the standard inner product by $\langle \x,\y\rangle$, and use the convention that $\x,\y\mapsto \langle \x,\y\rangle$ is conjugate-linear in its second argument in the case $\F=\C$. A matrix of size $m \times n$ where all entries are equal to 0 is denoted by $\zero_{m \times n}.$  The  $n \times n$ identity matrix is denoted by $\Id_n$. When the dimensions of the matrix are clear from the context, we may drop the subscript and simply write $\zero$ or $\Id$. 
The set of  $n\times n$ matrices with entries in $\F$ is denoted by $M_n(\F)$. The set of  $n \times n$  Hermitian matrices is denoted by $M_n(\F)^\sa,$ and the set of $n \times n$  Hermitian positive semidefinite matrices is denoted by $M_n(\F)^+.$ Note that a positive semidefinite matrix is necessarily a Hermitian matrix.   We denote the null space (or kernel)  of a matrix $A$ by $\nullspace(A)$  and the column space of $A$ by $\col(A)$. 

In several proofs, we use the standard fact that $\col(A)^\perp = \nullspace(A^\ast)$, where the symbol $\perp$ denotes the \emph{orthogonal complement}. When $A$ is a Hermitian matrix, the previous equality  reduces to $ \col(A)^\perp=\nullspace(A)$. The definition of positive (semi)definiteness and Loewner order follow.
\begin{definition}[Positive definite and positive semidefinite]\cite[Section 7.1]{horn2012matrix}
Let $A \in M_n(\mathbb{\F})^{\sa}$ be a Hermitian matrix. We say that $A$ is positive definite, written $A > 0$, if
        \[
          \langle A\vv, \vv\rangle > 0 \quad \text{for all nonzero } \vv \in \F^n.
        \]
 We say that $A$ is positive semidefinite, written  $A \geq 0$, if
        \[
          \langle A\vv, \vv \rangle \geq 0 \quad \text{for all } \vv \in \F^n.
        \]
\end{definition} 

Frequently, we will use the standard fact that a positive semidefinite matrix $A$ has a matrix \emph{square root} $A^{1/2}$, which is the unique positive semidefinite matrix satisfying $A^{1/2}A^{1/2}=A$, and therefore also that $\langle A\vv,\vv\rangle = \|A^{1/2}\vv\|^2$ for all $\vv\in \F^n$. If $A$ is positive definite, and therefore invertible, $A^{-1/2}$ denotes the square root of the inverse matrix $A^{-1}$.

\begin{definition}[Loewner order]\cite[Def. 7.7.1]{horn2012matrix}
For two Hermitian matrices $A,B\in M_n(\F)^\sa$, we say that $A$ is less than or equal to $B$ in the \emph{Loewner order} if  $B-A$ is positive semidefinite. In this case, we write $A\le B$ or $B\ge A$. 
\end{definition}
From the definition of positive semidefiniteness, it follows that $A\le B$ if and only if $\langle A\vv,\vv\rangle \le \langle B\vv,\vv\rangle$ for all $\vv\in \F^n$. The following lemma is used in Section \ref{sec:main}. It is a standard result that can be found in several matrix algebra textbooks such as \cite{horn2012matrix}.

\begin{lemma}\cite[Section 7.7]{horn2012matrix}
\label{lem:ordering_null_space}
Let $P,Q\in M_n(\F)^\sa.$ If $\zero \leq P \leq Q$, then $\nullspace(Q)\subseteq \nullspace(P)$.
\end{lemma} 

\begin{proof}
Let $\vv \in\nullspace(Q)$. Then
\[
0 = \langle \zero \vv,\vv\rangle \le
\langle P\vv,\vv\rangle \le 
\langle Q\vv,\vv\rangle = 0.
\]
This implies  $\langle P\vv,\vv\rangle = 0$. Since $P$ is positive semidefinite, we must have  that $P\vv=\zero$. Therefore, $\vv\in \nullspace(P)$. 
\end{proof}

When dealing with real positive semidefinite matrices, the Loewner order has a simple  geometric interpretation.  The \emph{ellipsoid associated to  $A$} is denoted  by $E_A$ and defined by
\[
E_A=
\{\vv\in \R^n \mid \langle A\vv,\vv\rangle \le 1\}.
\]
The  Loewner order may be described in terms of the ellipsoids $E_A$ and $E_B$, by the equivalency
\begin{equation*} \label{eq:ellipsoidAB}
A \leq B \iff E_A \supseteq E_B.
\end{equation*}

Indeed, suppose that $A,B\in M_n(\R)^\sa$ are positive semidefinite. Then we have
\[
A\le B\iff
\langle A\vv,\vv\rangle \le
\langle B\vv,\vv\rangle \quad\text{for all }\vv\in \R^n.
\]
Since $\langle A\vv,\vv\rangle$ and $\langle B\vv,\vv\rangle$ are nonnegative, a rescaling argument shows that this is equivalent to
\[
\langle B\vv,\vv\rangle\le 1 \implies \langle A\vv,\vv\rangle\le 1 \quad\text{for all }\vv\in \R^n,
\]
i.e. $E_B\subseteq E_A$.

When  $A \in M_n(\R)^\sa$ is positive semidefinite but not positive definite,  the  ellipsoid $E_A$ is \emph{degenerate} meaning that the set is unbounded. Indeed, suppose $\lambda=0$ is an eigenvalue of $A$ with unit eigenvector $\uu$. For any $t \in \R$, consider the vector $\vv_t=t\uu.$ We obtain $$\langle A\vv_t, \vv_t\rangle= t^2\langle A\uu, \uu \rangle=0 \leq 1.$$ We get that $\vv_t$ is in $E_A$ for all $t \in \R.$ Therefore,  $E_A$ is unbounded. 

In general, an $n \times n$  positive definite matrix $A$ has a corresponding ellipsoid with semi-axes of length $1/\sqrt{\lambda_i}$ where $\lambda_i$ is an eigenvalue of $A$ for $i \in \{1, \dots, n\}$. The semi-axes are aligned with the corresponding eigenvectors of $\lambda_i.$

\begin{example}\label{eg:minimal_upper_bounds}
Let us now illustrate (minimal) upper bounds in the Loewner order. The following example is typical and key to the proof of the anti-lattice property in \cite{kadison1951}. Consider the diagonal projection matrices
\[
A_1=\begin{bmatrix}
    1 & 0 \\
    0 & 0 
\end{bmatrix},\qquad
A_2 =
\begin{bmatrix}
    0 & 0 \\
    0 & 1 
\end{bmatrix}.
\]
The identity matrix
\[
B = \Id = \begin{bmatrix}
    1 & 0 \\
    0 & 1
\end{bmatrix}
\]
is an upper bound in the Loewner order for the set $\{A_1,A_2\}$, because $B-A_1=A_2\ge \zero$ and $B-A_2=A_1\ge \zero$. In fact, $B$ is a minimal upper bound for $\{A_1,A_2\}$. Indeed, suppose
\[
X =
\begin{bmatrix}
    x & y \\
    \overline{y} & z
\end{bmatrix}
\]
is a Hermitian matrix that satisfies $X\ge A_1$, $X\ge A_2$, and $X\le B$. Because $B-X$, $X-A_1$, and $X-A_2$ all are positive semidefinite, their diagonal entries are nonnegative and this implies that $x=z=1$. Then $B-X$ is positive semidefinite with zero diagonal entries, and so we must have $B-X=\zero$, so $X=B$. This proves that $B$ is a minimal upper bound.
However, $B$ is not a unique \emph{minimum} upper bound. The matrix
\[
C =
\begin{bmatrix}
    2 & \sqrt{2} \\
    \sqrt{2} & 2
\end{bmatrix}
\]
is also an upper bound for $\{A_1,A_2\}$, and yet $C-B$ is neither positive nor negative semidefinite, so $B$ and $C$ are incomparable in the Loewner order. It is also possible to prove directly--or by using Proposition \ref{minimal condition} below, that $C$ is also a minimal upper bound for $\{A_1,A_2\}$.

This example illustrates that two or more incomparable matrices in the Loewner order have more than one minimal upper bound. In fact, the set of minimal upper bounds is always uncountable, which follows from \cite[Theorem 2.3 and Corollary 2.4]{stott2017minimal} or \cite[Theorem 3.1 and Corollary 3.4]{stott2016maximal}.

In Figure \ref{fig:upper_bounds_example}, we plot the corresponding ellipses for all matrices involved. Minimality of the upper bounds $B$ and $C$ is evidenced by the fact that no larger ellipse is contained in the intersection of the ellipses for $A_1$ and $A_2$. Incomparability of $B$ and $C$ is reflected in the fact that neither ellipse contains the other.

Because $A_1$ and $A_2$ are not definite, their ellipses are degenerate and unbounded. However, by adding any positive definite matrix--such as the identity $\Id$, to all matrices involved, which preserves the Loewner order, we can obtain a geometric picture involving only bounded ellipses.
\end{example}

\begin{figure}[ht]
\centering

\begin{subfigure}[t]{0.48\textwidth}
\centering
\begin{tikzpicture}[scale = \commonscale]
\begin{axis}[axis equal image,xmin=-2,xmax=2,ymin=-2,ymax=2,xtick distance=1,xticklabel=\empty,ytick distance = 1,yticklabel=\empty]
\addplot[thick,blue,domain=-4:4,samples=5,smooth] ({-1},{\x});
\addplot[thick,blue,domain=-4:4,samples=5,smooth] ({1},{\x});
\addplot[thick,blue,domain=-4:4,samples=5,smooth] ({\x},{-1});
\addplot[thick,blue,domain=-4:4,samples=5,smooth] ({\x},{1});
\addplot[thick,purple,domain=0:2*pi,samples=50,smooth,fill=purple,fill opacity = 0.05] ({ellipsex(1,0,1,\x)},{ellipsey(1,0,1,\x)});
\end{axis}
\end{tikzpicture}
\end{subfigure}\hfill
\begin{subfigure}[t]{0.48\textwidth}
\centering
\begin{tikzpicture}[scale = \commonscale]
\begin{axis}[axis equal image,xmin=-2,xmax=2,ymin=-2,ymax=2,xtick distance=1,xticklabel=\empty,ytick distance = 1,yticklabel=\empty]
\addplot[thick,blue,domain=-4:4,samples=5,smooth] ({-1},{\x});
\addplot[thick,blue,domain=-4:4,samples=5,smooth] ({1},{\x});
\addplot[thick,blue,domain=-4:4,samples=5,smooth] ({\x},{-1});
\addplot[thick,blue,domain=-4:4,samples=5,smooth] ({\x},{1});
\addplot[thick,purple,domain=0:2*pi,samples=50,smooth,fill=purple,fill opacity = 0.05] ({ellipsex((1/sqrt(2+sqrt(2))+1/sqrt(2-sqrt(2)))/2,
(1/sqrt(2+sqrt(2))-1/sqrt(2-sqrt(2)))/2,(1/sqrt(2+sqrt(2))+1/sqrt(2-sqrt(2)))/2,\x)},{ellipsey((1/sqrt(2+sqrt(2))+1/sqrt(2-sqrt(2)))/2,
(1/sqrt(2+sqrt(2))-1/sqrt(2-sqrt(2)))/2,(1/sqrt(2+sqrt(2))+1/sqrt(2-sqrt(2)))/2,\x)});
\end{axis}
\end{tikzpicture}
\end{subfigure}

\vspace{1em}

\begin{subfigure}[t]{0.48\textwidth}
\centering
\begin{tikzpicture}[scale = \commonscale]
\begin{axis}[axis equal image,xmin=-1.5,xmax=1.5,ymin=-1.5,ymax=1.5,xtick distance=1,xticklabel=\empty,ytick distance = 1,yticklabel=\empty]
\addplot[thick,blue,domain=0:2*pi,samples=50,smooth] ({ellipsex(1/sqrt(2),0,1,\x)},{ellipsey(1/sqrt(2),0,1,\x)});
\addplot[thick,blue,domain=0:2*pi,samples=50,smooth] ({ellipsex(1,0,1/sqrt(2),\x)},{ellipsey(1,0,1/sqrt(2),\x)});
\addplot[thick,purple,domain=0:2*pi,samples=50,smooth,fill=purple,fill opacity = 0.05] ({ellipsex(1/sqrt(2),0,1/sqrt(2),\x)},{ellipsey(1/sqrt(2),0,1/sqrt(2),\x)});
\end{axis}
\end{tikzpicture}
\end{subfigure} \hfill
\begin{subfigure}[t]{0.48\textwidth}
\centering
\begin{tikzpicture}[scale = \commonscale]
\begin{axis}[axis equal image,xmin=-1.5,xmax=1.5,ymin=-1.5,ymax=1.5,xtick distance=1,xticklabel=\empty,ytick distance = 1,yticklabel=\empty]
\addplot[thick,blue,domain=0:2*pi,samples=50,smooth] ({ellipsex(1/sqrt(2),0,1,\x)},{ellipsey(1/sqrt(2),0,1,\x)});
\addplot[thick,blue,domain=0:2*pi,samples=50,smooth] ({ellipsex(1,0,1/sqrt(2),\x)},{ellipsey(1,0,1/sqrt(2),\x)});
\addplot[thick,purple,domain=0:2*pi,samples=50,smooth,fill=purple,fill opacity = 0.05] ({ellipsex((1/sqrt(3+sqrt(2))+1/sqrt(3-sqrt(2)))/2,
(1/sqrt(3+sqrt(2))-1/sqrt(3-sqrt(2)))/2,(1/sqrt(3+sqrt(2))+1/sqrt(3-sqrt(2)))/2,\x)},{ellipsey((1/sqrt(3+sqrt(2))+1/sqrt(3-sqrt(2)))/2,
(1/sqrt(3+sqrt(2))-1/sqrt(3-sqrt(2)))/2,(1/sqrt(3+sqrt(2))+1/sqrt(3-sqrt(2)))/2,\x)});
\end{axis}
\end{tikzpicture}
\end{subfigure} \hfill

\caption{Plots of the ellipses associated to the matrices in Example \ref{eg:minimal_upper_bounds}. \\\\
Top Left: The ellipses $E_{A_1}$ and $E_{A_2}$ associated to $A_1$ and $A_2$ are degenerate, and correspond to the strips $-1\le x\le 1$ and $-1\le y\le 1$ (boundaries drawn in blue). The identity matrix $B=\Id$ has the unit disk $E_B$ as its ellipse (shaded). Because the unit circle is contained in both strips, $B$ is an upper bound for $\{A_1,A_2\}$. Because $B$ is a minimal upper bound, there is no ellipse strictly larger than $E_B$ that is contained in $E_{A_1}\cap E_{A_2}$.\\\\
Top Right: The ellipse $E_C$ is also a maximal ellipse contained in $E_{A_1}\cap E_{A_2}$, reflecting the fact that $C$ is also a minimal upper bound for $\{A_1,A_2\}$. The ellipse $E_C$ is not contained in and does not contain $E_B$, reflecting the fact that $B$ and $C$ are incomparable in the Loewner order.\\\\
Bottom panels: The ellipses corresponding to $A_1+\Id$, $A_2+\Id$ (unfilled), $B+\Id$ (left panel), and $C+\Id$ (right panel). All the same features as in the top panels are present, but by adding the identity matrix to all matrices involved, we obtain a picture involving only bounded nondegenerate ellipses.}
\label{fig:upper_bounds_example}
\end{figure}

In general, when finding upper bounds for a finite set $\{A_1,\ldots,A_k\}$ of Hermitian matrices, by adding a sufficiently large multiple of the identity matrix (larger than the most negative eigenvalue of any $A_i$), we may assume without loss of generality that all matrices involved are positive definite. Positive definiteness is not necessary in our main algorithm below, but it allows for a geometric picture in terms of bounded ellipsoids.

We are now ready to introduce the main results of this paper.

\section{Main Results} \label{sec:main}

In this section, we introduce an algorithm (Algorithm \ref{main_algorithm}) that outputs  an exact  minimal upper bound for any (non-empty) finite set of $n\times n$ Hermitian matrices. The algorithm takes as  input any upper bound for the  finite set. 
Theorem \ref{thm:algorithm_terminates} represents  one of the main results of this paper. It shows  that the algorithm ends in at most $n$ iterations. Before introducing the algorithm,  we introduce the theoretical results needed to show the exactness of the algorithm. 

The key idea in the algorithm that follows is to produce a decreasing sequence of upper bound matrices by subtracting off a positive definite rank one matrix at each step. Every rank one positive definite matrix is of the form $\lambda \e \e^\ast$, where $\e$ is a unit vector and $\lambda>0$. Given a positive semidefinite matrix $P\ge \zero$, the next lemma identifies exactly which nonnegative scalars $\lambda$ satisfy $\lambda \e\e^\ast\le P$. It states that the set of such $\lambda$ forms a closed interval $[0,\lambda_{\max}]$, where $\lambda_{\max}>0$ if and only if $\e$ is in the column space $\col(P)$.

\begin{lemma}
\label{lem:lambda}
Let $P \in M_n(\mathbb{F})^+$  be a positive semidefinite matrix. Let $\e \neq \zero$ be in $\F^n$. Consider the associated rank one Hermitian matrix $\e\e^\ast$. Then the following statements hold:
\begin{enumerate}[label=(\arabic*)]
\item There exists $\lambda>0$ such that $\lambda \e \e^\ast \leq P$ if and only if $\e\in \col(P)$.

\item If (1) holds, then the maximum value of $\lambda>0$ satisfying $\lambda \e \e^* \le P$ is $$\lambda = \frac{1}{\langle P\uu,\uu \rangle} \quad  \text{for any }\uu\in \F^n  \text{ such that }  P\uu=\e.$$

\item If $\e\in \col(P)$ and $\lambda =\max\{ \mu  > 0 \mid \mu \e\e^* \le P \}$, then $\e \not\in \col(P - \lambda  \e \e^*)$.

\end{enumerate}
\end{lemma}

\begin{proof}
(1) Suppose that $\e \not \in \text{col}(P)$ and let $\lambda>0$. Since $\text{col}(P) = \text{null}(P)^\perp$, we must have that $\e$ is not orthogonal to every vector in null$(P)$ and so there is a vector $\x \in \text{null}(P) \text{ such that } \langle \x, \e \rangle \ne 0$. Therefore,
\[
\langle (P-\lambda \e \e^\ast)\x,\x\rangle =
0-\lambda \langle (\e\e^\ast)\x,\x\rangle = 
-\lambda |\langle \x,\e\rangle|^2.
\]
Since $\lambda>0$, this shows that $P-\lambda \e\e^\ast$ is not positive semidefinite. 

Conversely, suppose that $\e \in \col(P)$. Let $\uu \in \mathbb{F}^n$ satisfy $P\uu = \e$, and note that since $\e \not = \zero$, we also have $\uu \not = \zero$ and $\langle P\uu,\uu\rangle = \|P^{1/2}\uu\|^2 >0$. Let $\vv \in \mathbb{F}^n$. We know that $ \langle (\e\e^*)\vv,\vv \rangle = | \langle \e, \vv \rangle |^2 $. Then
\begin{align*}
\langle (\e\e^\ast)\vv,\vv\rangle =
|\langle \e,\vv\rangle|^2 =
|\langle P\uu,\vv\rangle|^2 \le
\langle P\uu,\uu\rangle \langle P\vv,\vv\rangle,
\end{align*}
where the last step follows from the Cauchy-Schwarz inequality applied to the positive semidefinite sesquilinear form $(\x,\y)\mapsto \langle P\x,\y\rangle$. This shows that $$\frac{1}{\langle P\uu, \uu \rangle}\langle (\e\e^*)\vv,\vv \rangle \le \langle P\vv, \vv \rangle$$ and because $\vv$ was arbitrary, we have $\frac{1}{\langle P\uu, \uu \rangle} \e \e^* \le P $. 

Note that this proof actually shows that $\langle P\uu,\uu\rangle$ has the same value for every vector $\uu\in \F^n$ such that $P\uu=\e$. This can be also proved directly. If $\uu,\vv\in \F^n$ satisfy $P\uu=P\vv=\e$, then 
\[
\langle P\uu,\uu\rangle =
\langle P\vv,\uu\rangle =
\langle \vv,P\uu\rangle =
\langle \vv,P\vv\rangle =
\langle P\vv,\vv\rangle.
\]

\noindent (2) Suppose that $\e \in \col(P)$ and let $\uu$ be any vector satisfying $P\uu=\e$. Hence, $\langle P\uu,\uu\rangle \ne 0$. Let $\lambda > 0$ be any positive number that satisfies $\lambda \e\e^\ast \le P$. We have
\[
\langle P\uu,\uu\rangle \ge
\langle \lambda (\e\e^\ast)\uu,\uu\rangle =
\lambda |\langle \uu,\e\rangle|^2 =
\lambda \langle P\uu,\uu\rangle^2.
\]
Therefore, $\lambda \le \frac{1}{\langle P\uu,\uu\rangle}$. Since  $\frac{1}{\langle P\uu,\uu\rangle}\e\e^\ast \le P,$ we obtain
\[
\frac{1}{\langle P\uu,\uu\rangle} =
\max\{\lambda>0\mid \lambda \e \e^\ast \le P\}.
\]

\noindent (3) Suppose that $\mu>0$ satisfies $\mu \e\e^\ast \le P$, and that $\e \in \col(P-\mu \e \e^\ast)$. Since $\e\in \col(P-\mu \e\e^\ast)$, by (1) there exists $\lambda>0$ such that $\lambda \e\e^\ast \le P-\mu \e\e^\ast$. Rearranging gives $(\mu + \lambda)\e\e^\ast \le P$. Hence,  $\mu$ is not maximal. 
\end{proof}

Now, to identify minimal upper bounds, we first need to generalize Lemma \ref{lem:ordering_null_space} to the setting of two comparable upper bounds for the same set of matrices.

\begin{lemma}
\label{null_space_proof}
Let $A_1,\ldots,A_k,B,C\in M_n(\F)^\sa$ be Hermitian matrices. If  $A_i\le C\le B$ for each $i \in \{1, \dots, k\},$ then $$ \sum_{i=1}^k \nullspace(B-A_i) \subseteq \nullspace(B-C).$$
\end{lemma}

\begin{proof}
For each $i$, we have $$\zero\le B-C\le B-A_i.$$ By Lemma  \ref{lem:ordering_null_space}, $   \nullspace(B-A_i)\subseteq \nullspace(B-C)$ for all $i \in \{1, \dots, k\}.$ Since $\nullspace(B-C)$ is a linear subspace, we must have 
\[
\sum_{i=1}^k \nullspace(B-A_i)\subseteq \nullspace(B-C). \qedhere\]
\end{proof}

Next we give an independent proof of an algebraic characterization due to Stott \cite[Theorem 1.15]{stott2017minimal}. It provides a sufficient and necessary  condition for an upper bound to be  minimal.
\begin{proposition}\cite[Theorem 1.15]{stott2017minimal}
\label{minimal condition}
Suppose $A_1, \cdots,A_k,B \in M_n(\mathbb{F})^\sa$ are Hermitian matrices, and that $B\ge A_i$ for all $i$, so that $B$ is an upper bound for $\{A_1,\ldots,A_k\}$ in the Loewner order. Let $E = \sum_{i = 1}^{k} \text{null}(B-A_i) $. The upper bound $B$ is minimal if and only if $E=\mathbb{F}^n$.
\end{proposition}

\begin{proof}
Suppose that $E=\F^n$ and that $C\le B$ is also an upper bound of $A_1,\ldots,A_k$.  We have $A_i\le C\le B$ for all $i$. By Lemma \ref{null_space_proof}, we have $E=\F^n\subseteq \nullspace(B-C)$, and so $B-C=\zero$, hence $C=B$. This proves that $B$ is a minimal upper bound.

Conversely, suppose that $E\ne \F^n$. We will prove that $B$ is not minimal. Since $E\ne \F^n$, $E^\perp \ne \{0\}$. Choose a vector $\e \in E^\perp = \bigcap_{i=1}^k \col(B-A_i)$ with $\e\ne \zero$. By Lemma \ref{lem:lambda}.(1), there exist $\lambda_1,\ldots,\lambda_k>0$ such that $\lambda_i \e\e^\ast \le B - A_i$ for each $i$. Set $\lambda = \min\{\lambda_1,\ldots,\lambda_n\}$. Then $\lambda \e\e^\ast \le \lambda_i \e\e^\ast \le B-A_i$ for each $i$, and so
\[
A_i \le 
B-\lambda \e\e^\ast
\]
for all $i$. Therefore $B-\lambda \e\e^\ast\le B$ is an upper bound for $A_1,\ldots,A_k$ that is strictly smaller than $B$, so $B$ is not minimal.
\end{proof}

Equivalently, $E=\F^n$ if and only if $E^\perp = \bigcap_{i=1}^k \col(B-A_i)=\{\zero\}$. The subspace $E$ is the span of all those vectors $\vv$ that satisfy equality $B\vv=A_i\vv$ for at least one $i$. Therefore Proposition \ref{minimal condition} states that an upper bound is minimal when $B$, viewed as a linear transformation, agrees with at least one $A_i$ on enough vectors to span the whole space.

Proposition \ref{minimal condition} motivates the construction of the following algorithm. Given an upper bound $B\ge A_1,\ldots, A_k$ that is not minimal, the proof in fact showed explicitly how to find a smaller upper bound $B-\lambda \e \e^\ast$ by subtracting off a rank one matrix with range spanned by a vector $\e\in E^\perp$. This gives an iterative method to improve any non-minimal upper bound to a smaller one, which is the basis of the following algorithm. 

\begin{algorithm}\label{main_algorithm}
Provided a finite set of $n\times n$ Hermitian matrices $\{A_1, ..., A_k \}$ and some upper bound $B$ for $\{A_1,\ldots,A_k\}$. Starting with $B_0 := B,$ the following algorithm recursively produces a sequence of upper bounds $B_0 \ge B_1 \ge B_2 \ge \cdots \ge B_m$ such that $B_m$ is a minimal upper bound, and $m\le n$ (see Theorem \ref{thm:algorithm_terminates}).

Recursively, for each $r=0,1,2,\ldots$, assuming that $B_r$ has been computed:

\begin{enumerate}
\item Calculate $\Delta_i = B_r - A_i$ for all $i$.

\item Then, check whether $E_r:= \sum_{i=0}^k\nullspace(\Delta_i)=\F^n$, or not. If $E_r=\F^n$, then $B_r$ is a minimal upper bound and the algorithm terminates. Otherwise:

\item Arbitrarily choose any nonzero vector $\e\in E_r^\perp=\bigcap_{i=1}^k \col(\Delta_i)$. 

(Optionally, to reduce numerical overflow errors, normalize $\e$.)

\item For each $i$, find a vector $\uu_i\in \F_n$ with $\Delta_i \uu_i=\e$, and set 
\[
\lambda_i=\frac{1}{\langle \Delta_i \uu_i,\uu_i\rangle} = \frac{1}{\langle \e,\uu_i\rangle} \quad\left(=\max \{\mu >0\mid \mu \e\e^\ast \le \Delta_i\}\right).
\]

\item Set $\lambda := \min\{\lambda_1,\ldots,\lambda_k\}$, then set $B_{r+1}:=B_r-\lambda \e \e^\ast$, and recurse.
\end{enumerate}
\end{algorithm}

Note that given any such set $\{A_1,\ldots,A_k\}$, it is always possible to find an upper bound $B$ of the form $c\Id$, where $c>0$ is a scalar. For example, it suffices to choose $c\ge\max\{\lambda_{\max}(A_1),\ldots,\lambda_{\max}(A_k)\}$, where $\lambda_{\max}(A_i)$ denotes the largest eigenvalue of $A_i$.

Every step in Algorithm \ref{main_algorithm} can be implemented with standard methods. For example, in Step 2, one may use singular value decomposition to compute orthonormal bases for each null space $\nullspace(\Delta_i)$. If the union of these bases do not span $\R^n$, then singular value decomposition can be used again to find an orthonormal basis for $E_r^\perp$, from which $\e$ can be selected arbitrarily.

To avoid numerical issues in Step 4 when $\Delta_i$ is singular (or nearly singular), a robust approach is to set $\uu_i = \Delta_i^\dagger \e$, where $\Delta_i^\dagger$ denotes the Moore-Penrose pseudoinverse of $\Delta_i.$

Note that while we framed the algorithm in terms of finding minimal upper bounds, it can also be used to calculate maximal lower bounds in the Loewner order. Indeed, $B$ is a (maximal) lower bound for $\{A_1,\ldots,A_k\}$ if and only if $-B$ is a (minimal) upper bound for $\{-A_1,\ldots,-A_k\}$, and so one needs only to negate all matrices involved.

We now prove that no matter how the vector $\e$ in Step 3 is chosen, the algorithm will always converge to a minimal upper bound in $n$ steps or less. However, different choices of $\e$ at each step will result in different minimal upper bounds produced by this algorithm (see Example \ref{2x2_example} and Remark \ref{rem:every_minimal_bound?} below).

\pagebreak
\begin{pseudocode}
What follows is pseudocode for a recursive implementation of Algorithm \ref{main_algorithm}.

\label{main_algorithm_code}
\begin{algorithmic}
\BeginBox
\Function{MinUpperBoundBelow}{$B$, $S=\{A_1,\ldots,A_k\}$}
    \LComment{Given an upper bound $B$ for $\{A_1,\ldots,A_k\}$, recursively computes a minimal upper bound for $\{A_1,\ldots,A_k\}$ which is dominated by $B$.}
    \For{$i=1,\ldots,k$}
        \State $\Delta_i\gets B-A_i$
        \State $N_i\gets \text{nullbasis}(\Delta_i)$
        \Comment{cols. of $N_i$ are an ONB for $\nullspace(\Delta_i)$}
    \EndFor
    \State $N\gets [N_1\;\cdots \; N_k]$
    \If{$\text{rank}(S)=n$}
        \State \Return B
    \Else
        \State $e\gets \text{nullbasis}(N^\ast)[1]$
        \Comment{Choose first col. of ONB for $\nullspace(N^\ast)=E_r^\perp$.}
        \For{$i=1,\ldots,k$}
            \State $\lambda_i\gets 1/\langle \Delta_i^\dagger \e,\e\rangle$.
        \EndFor
        \State $\lambda \gets \min\{\lambda_1,\ldots,\lambda_k\}$.
        \State\Return MinUpperBoundBelow($B-\lambda \e\e^\ast,S$).
    \EndIf
\EndFunction

\Function{MinUpperBound}{$S=\{A_1,\ldots,A_k\}$}
    \LComment{Returns a minimal upper bound for $\{A_1,\ldots,A_k\}$ whose largest eigenvalue does not exceed all eigenvalues of all $A_i$}
    \For{$i=1,\ldots,k$}
        \State $M_i\gets \lambda_{\max}(A_i)$
        \Comment{Maximum eigenvalue of $A_i$}
    \EndFor
    \State \Return MinUpperBoundBelow($\max\{M_1,\ldots,M_k\}\Id$,$S$)
\EndFunction
\EndBox
\end{algorithmic}
\end{pseudocode}

Our next goal is to prove that the algorithm in fact terminates at a minimal upper bound in finitely many steps. The first key observation are that the spaces $E_r$ are an increasing chain:
\[
E_0\subseteq E_1\subseteq E_2\subseteq \cdots.
\]
The second key observation is that given any choice of nonzero $\e \in E_r^\perp$, as long as $\lambda$ is chosen optimally, then $E_r\ne E_{r+1}$ at each step. The next proposition presents these details in generality.

\begin{proposition}
\label{prop:dimension_decreases}
Let $B, A_1,\ldots, A_k$ be Hermitian matrices. Suppose that $B$ is an upper bound for $\{A_1,\ldots, A_k\}$, and let $E = \sum_{i=1}^k \nullspace(B-A_i).$ Suppose that $\e\in E^\perp$ is nonzero, and $\lambda$ is chosen optimally in the sense that
\begin{equation}\label{eq:optimal_lambda}
\lambda = \min_{i=1,\ldots,k} \max \{\mu\ge 0 \mid \mu \e\e^\ast \le B-A_i\}>0.
\end{equation}
Define $C=B-\lambda \e\e^\ast$ and $F=\sum_{i=1}^k \nullspace(C-A_i)$. Then $C$ is an upper bound for $\{A_1,\ldots,A_k\}$ and $E\subsetneq F$.
\end{proposition}

\begin{proof}
We have $\lambda \e\e^\ast \le B-A_i$ for each $i$, and so $C\ge A_i$ for each $i$, meaning that $C$ is an upper bound. Therefore $\zero \le C-A_i\le B-A_i$, and so Lemma \ref{lem:ordering_null_space} implies that $E\subseteq F$. 

\medskip

Now, since
\[
\e \in E^\perp = \bigcap_{i=1}^k \col(B-A_i),
\]
we have that $\lambda$ is strictly positive by Lemma \ref{lem:lambda}.(1). Moreover, Lemma \ref{lem:lambda}.(3) implies that $\e\not \in \col(B-A_i-\lambda \e\e^\ast)=\col(C-A_i)$ for some $i$, and so
\[
\e\not\in \bigcap_{i=1}^k \col(C-A_i)=F^\perp.
\]
Therefore $\e\in E^\perp\setminus F^\perp$, showing that $E^\perp\ne F^\perp$, so $E\ne F$.
\end{proof}

\begin{theorem}
\label{thm:algorithm_terminates}

If $\{A_1,\ldots,A_k\}$ is a set of $n\times n$ Hermitian matrices and $B\in M_n(\F)^\sa$ is an upper bound for $\{A_1,\ldots A_k\}$, then Algorithm \ref{main_algorithm} converges to a minimal upper bound for $\{A_1,\ldots,A_k\}$ less than $B$ in at most $n$ steps.
\end{theorem}

\begin{proof}
Algorithm \ref{main_algorithm} produces a successive sequence of Hermitian matrices $B_0\ge B_1\ge B_2\ge\cdots$ and associated subspaces $E_r = \sum_{i=0}^k \nullspace(B_r-A_i)$. Proposition \ref{prop:dimension_decreases} shows that each $B_r$ is still an upper bound for $\{A_1,\ldots A_k\}$, and that
\[
E_0\subsetneq E_1\subsetneq E_2\subsetneq\cdots
\]
is a strictly increasing chain of subspaces of $\F^n$. Since $\dim(\F^n)=n$, we must have $E_m=\F^n$ for some $m\le n$. By Proposition \ref{minimal condition},  $B_m$ is a minimal upper bound.
\end{proof}

\begin{remark}
At each step, $B_{r+1}=B_r-\lambda \e\e^\ast$ is obtained by subtracting a rank one matrix from $B_r$. It follows that $\dim(E_r)$ increases by exactly one at each step. Therefore, one can show that Algorithm \ref{main_algorithm} terminates in exactly $m$ steps, where $m=n-\dim(E_0)$. In particular, if each $\Delta_i=B_0-A_i$ is strictly positive definite, then $\dim(E_0)=0$ and the algorithm takes exactly $n$ steps. 
\end{remark}

We now provide examples showing how Algorithm \ref{main_algorithm}. We begin with an example using two $2 \times 2$ matrices.

\begin{example}
\label{2x2_example} Let
\[
B =
\begin{bmatrix}
8 & 2 \\
2 & 8
\end{bmatrix}, \quad
A_1 =
\begin{bmatrix}
3 & 1 \\
1 & 3
\end{bmatrix},\quad\text{and}\quad
A_2 =
\begin{bmatrix}
5 & 1 \\
1 & 1
\end{bmatrix},
\]
so that $B$ is an upper bound for $\{A_1,A_2\}$.

\medskip

In the first iteration of the algorithm, we find that
\[
\Delta_1 = 
\begin{bmatrix}
    5 & 1 \\
    1 & 5
\end{bmatrix} \quad\text{and}\quad
\Delta_2 =
\begin{bmatrix}
    3 & 1 \\
    1 & 7
\end{bmatrix}
\]
are both strictly positive definite. So, $E_0 =\nullspace(\Delta_1)+\nullspace(\Delta_2)=\{\zero\}$. We then choose a nonzero vector $\e\in E_0^\perp$ arbitrarily. Here we choose $\e=[1\; 1]^\top$. Then, we compute
\begin{align*}
\uu_1 =
\Delta_1^{-1}\e &= \frac{1}{6}\begin{bmatrix}
    1 \\ 1
\end{bmatrix} \quad&\text{so}\quad
\lambda_1 =
\frac{1}{\langle \e,\uu_1\rangle} &=
3, \\
\uu_2 =
\Delta_2^{-1}\e &= \frac{1}{10}\begin{bmatrix}
    3 \\ 1
\end{bmatrix} \quad&\text{so}\quad
\lambda_2 =
\frac{1}{\langle \e,\uu_2\rangle} &=
\frac{5}{2}.
\end{align*}
The optimal choice of scalar is $\lambda = \min\{3,\frac{5}{2}\}=\frac{5}{2}$. Therefore we obtain the new upper bound
\[
B_1 =
B-\lambda \e\e^\ast =
\frac{1}{2}\begin{bmatrix}
    \ph 11 & -1 \\
    -1 & \ph 11
\end{bmatrix}.
\]

\medskip

Now, the algorithm repeats with $B_1$, and we re-calculate
\[
\Delta_1 =
B_1-A_1 =
\frac{1}{2}\begin{bmatrix}
    \ph 5 & -3 \\
    -3 & \ph 5
\end{bmatrix}\quad\text{and}\quad
\Delta_2 =
B_1-A_2 =
\frac{1}{2}\begin{bmatrix}
    \ph 1 & -3 \\
    -3 & \ph 9
\end{bmatrix}.
\]
As promised, $\Delta_1,\Delta_2\ge \zero$, confirming that $B_1$ is in fact an upper bound for $\{A_1,A_2\}$. Whereas $\Delta_1$ is positive definite, $\Delta_2$ is now singular and
\[
E_1 =
\nullspace(\Delta_1)+\nullspace(\Delta_2) =
\nullspace(\Delta_2) =
\spn\left\{
\begin{bmatrix}
    3 \\ 1
\end{bmatrix}\right\}.
\]
As expected, $E_1\supsetneq E_0$, and now we choose a nonzero vector $\e=[-1\; 3]^\top\in E_1^\perp$. (Note that in this case, $E_1^\perp$ is one dimensional, and so the choice of $\e$ is unique up to a scalar multiple, which will not affect the resulting minimal upper bound $B_2$ computed at the next step.) Then we find
\begin{align*}
\uu_1 =
\Delta_1^{-1}\e &= \frac{1}{2}\begin{bmatrix}
    1 \\ 3
\end{bmatrix} \quad&\text{so}\quad
\lambda_1 =
\frac{1}{\langle \e,\uu_1\rangle} &=
\frac{1}{4}, \\
\uu_2 = 
\Delta_2^\dagger \e &= \frac{1}{5}\begin{bmatrix}
    -1 \\ \ph 3
\end{bmatrix} \quad&\text{so}\quad
\lambda_2 =
\frac{1}{\langle \e,\uu_2\rangle} &=
\frac{1}{2}.
\end{align*}
The optimal scalar is $\lambda = \min\{\frac{1}{4},\frac{1}{2}\}=\frac{1}{4}$.  (Note that while the choice of $\uu_2$ satisfying $\Delta_2 \uu_2=\e$ was not unique, any other choice such as $\uu_2' = [-2 \; 0]^\top$ will result in the same value for $\lambda_2$.)  Now, the next upper bound is
\[
B_2 =
B_1 -\lambda \e\e^\ast = 
\frac{1}{4}\begin{bmatrix}
    21 & 1 \\
    1 & 13
\end{bmatrix}.
\]

\medskip

Next the algorithm repeats with $B_2$. We obtain
\[
\Delta_1 =
B_2 -A_1 =\frac{1}{4}\begin{bmatrix}
    \ph 9 & -3 \\
    -3 & \ph 1
\end{bmatrix}\quad\text{and}\quad
\Delta_2 =
B_2 - A_2 =
\frac{1}{4}\begin{bmatrix}
    \ph 1 & -3 \\
    -3 & \ph 9
\end{bmatrix}.
\]
Now both $\Delta_1$ and $\Delta_2$ are singular, and we find
\[
E_2 =
\nullspace(\Delta_1)+\nullspace(\Delta_2) =
\spn\left\{\begin{bmatrix}
    1 \\ 3
\end{bmatrix}\right\} +
\spn\left\{\begin{bmatrix}
    3 \\ 1
\end{bmatrix}\right\} =
\F^2.
\]
Proposition \ref{minimal condition} now shows that $B_2$ is a minimal upper bound for $\{A_1,A_2\}$, and indeed the algorithm has converged in exactly $n=2$ steps.

The steps of the algorithm in this example are plotted using the associated ellipses in Figure \ref{fig:2x2_example_ellipses}. Here, the decreasing sequence of positive definite upper bounds corresponds to an increasing sequence of ellipses.

In the same example, if we instead make a different choice of the vector $\e$ at each step in the algorithm, then a different minimal upper bound is obtained. Using the same matrices $B,A_0,A_1$ as defined in Example \ref{2x2_example}, if we instead successively choose $\e=[1\; -1]^\top$ and $\e=[1 \; 2]^\top$ at each step, then we obtain the upper bounds
\[
B_1' = \frac{1}{3}\begin{bmatrix}
    19 & 11\\
    11 & 19
\end{bmatrix},\qquad
B_2'= 
\frac{1}{3}\begin{bmatrix}
    17 & 7\\
    7 & 11
\end{bmatrix}
\]
for $\{A_1,A_2\}$, with $B_2$ being  a minimal upper bound that is different from the one obtained in Example \ref{2x2_example}. Furthermore, if instead we take $\e=[-1\; 3]^\top$ at the first step, and $\e=[2\; 1]^\top$ at the second step, then we obtain the upper bounds
\[
B_1'' =
\frac{1}{7}\begin{bmatrix}
    53 & 23 \\
    23 & 29
\end{bmatrix},\qquad
B_2'' = B_2' =
\frac{1}{3}\begin{bmatrix}
    17 & 7 \\
    7 & 11
\end{bmatrix}.
\]
Interestingly, a different choice of the vector $\e$ in this case still produces the same minimal upper bound.
\end{example}

\begin{remark}
\label{rem:every_minimal_bound?}
While the previous example show that different choices of the vector $\e$ at each step can produce different minimal upper bounds, it is unclear if it is possible to obtain \emph{every} minimal upper bound for $\{A_1,\ldots,A_k\}$ less than $B$ by making a suitable choice of $\e$ at each step of the algorithm.
\end{remark}

\renewcommand{\commonscale}{0.75}
\newcommand{\commonxmin}{-1}
\newcommand{\commonxmax}{1}
\newcommand{\commonymin}{-1.2}
\newcommand{\commonymax}{1.2}
\begin{figure}[ht]
\centering

\begin{subfigure}[t]{0.48\textwidth}
\centering

\begin{tikzpicture}[scale = \commonscale]
\begin{axis}[xmin=\commonxmin,xmax=\commonxmax,ymin=\commonymin,ymax=\commonymax,xtick distance=0.5,xticklabel=\empty,ytick distance = 0.5,yticklabel=\empty]
\addplot[thick,blue,domain=0:2*pi,samples=50,smooth] ({ellipsex(A1_11,A1_12,A1_22,\x)},{ellipsey(A1_11,A1_12,A1_22,\x)});
\addplot[thick,red,domain=0:2*pi,samples=50,smooth] ({ellipsex(A2_11,A2_12,A2_22,\x)},{ellipsey(A2_11,A2_12,A2_22,\x)});
\addplot[thick,purple,domain=0:2*pi,samples=50,smooth,fill=purple,fill opacity = 0.05] ({ellipsex(B0_11,B0_12,B0_22,\x)},{ellipsey(B0_11,B0_12,B0_22,\x)});
\end{axis}
\end{tikzpicture}

\end{subfigure}

\begin{subfigure}[t]{0.48\textwidth}
\centering

\begin{tikzpicture}[scale = \commonscale]
\begin{axis}[xmin=\commonxmin,xmax=\commonxmax,ymin=\commonymin,ymax=\commonymax,xtick distance=0.5,xticklabel=\empty,ytick distance = 0.5,yticklabel=\empty]
\addplot[thick,blue,domain=0:2*pi,samples=50,smooth] ({ellipsex(A1_11,A1_12,A1_22,\x)},{ellipsey(A1_11,A1_12,A1_22,\x)});
\addplot[thick,red,domain=0:2*pi,samples=50,smooth] ({ellipsex(A2_11,A2_12,A2_22,\x)},{ellipsey(A2_11,A2_12,A2_22,\x)});
\addplot[thick,purple,domain=0:2*pi,samples=50,smooth,fill=purple,fill opacity = 0.05] ({ellipsex(B1_11,B1_12,B1_22,\x)},{ellipsey(B1_11,B1_12,B1_22,\x)});
\end{axis}
\end{tikzpicture}

\end{subfigure} \hfill
\begin{subfigure}[t]{0.48\textwidth}
\centering

\begin{tikzpicture}[scale = \commonscale]
\begin{axis}[xmin=\commonxmin,xmax=\commonxmax,ymin=\commonymin,ymax=\commonymax,xtick distance=0.5,xticklabel=\empty,ytick distance = 0.5,yticklabel=\empty]
\addplot[thick,blue,domain=0:2*pi,samples=50,smooth] ({ellipsex(A1_11,A1_12,A1_22,\x)},{ellipsey(A1_11,A1_12,A1_22,\x)});
\addplot[thick,red,domain=0:2*pi,samples=50,smooth] ({ellipsex(A2_11,A2_12,A2_22,\x)},{ellipsey(A2_11,A2_12,A2_22,\x)});
\addplot[thick,purple,domain=0:2*pi,samples=50,smooth,fill=purple,fill opacity = 0.05] ({ellipsex(B1alt_11,B1alt_12,B1alt_22,\x)},{ellipsey(B1alt_11,B1alt_12,B1alt_22,\x)});
\end{axis}
\end{tikzpicture}
\end{subfigure} 

\begin{subfigure}[t]{0.48\textwidth}
\centering

\begin{tikzpicture}[scale = \commonscale]
\begin{axis}[xmin=\commonxmin,xmax=\commonxmax,ymin=\commonymin,ymax=\commonymax,xtick distance=0.5,xticklabel=\empty,ytick distance = 0.5,yticklabel=\empty]
\addplot[thick,blue,domain=0:2*pi,samples=50,smooth] ({ellipsex(A1_11,A1_12,A1_22,\x)},{ellipsey(A1_11,A1_12,A1_22,\x)});
\addplot[thick,red,domain=0:2*pi,samples=50,smooth] ({ellipsex(A2_11,A2_12,A2_22,\x)},{ellipsey(A2_11,A2_12,A2_22,\x)});
\addplot[thick,purple,domain=0:2*pi,samples=50,smooth,fill=purple,fill opacity = 0.05] ({ellipsex(B2_11,B2_12,B2_22,\x)},{ellipsey(B2_11,B2_12,B2_22,\x)});
\end{axis}
\end{tikzpicture}

\end{subfigure} \hfill
\begin{subfigure}[t]{0.48\textwidth}
\centering

\begin{tikzpicture}[scale = \commonscale]
\begin{axis}[xmin=\commonxmin,xmax=\commonxmax,ymin=\commonymin,ymax=\commonymax,xtick distance=0.5,xticklabel=\empty,ytick distance = 0.5,yticklabel=\empty]
\addplot[thick,blue,domain=0:2*pi,samples=50,smooth] ({ellipsex(A1_11,A1_12,A1_22,\x)},{ellipsey(A1_11,A1_12,A1_22,\x)});
\addplot[thick,red,domain=0:2*pi,samples=50,smooth] ({ellipsex(A2_11,A2_12,A2_22,\x)},{ellipsey(A2_11,A2_12,A2_22,\x)});
\addplot[thick,purple,domain=0:2*pi,samples=50,smooth,fill=purple,fill opacity = 0.05] ({ellipsex(B2alt_11,B2alt_12,B2alt_22,\x)},{ellipsey(B2alt_11,B2alt_12,B2alt_22,\x)});
\end{axis}
\end{tikzpicture}
\end{subfigure} 

\caption{Top-center and left panels: Plots of the ellipses $E_{A_1}$ (red vertically oriented ellipse), $E_{A_2}$ (blue horizontally oriented ellipse), and $E_{B_i}$ (inner shaded ellipse) for $i=0,1,2$ from Example \ref{2x2_example}. Since $B_0\ge B_1\ge B_2$, the ellipse $E_{B_i}$ becomes larger at each step. Because $B_2$ is a minimal upper bound for $A_1,A_2$, there is no larger ellipse than $E_{B_2}$ which is still contained in $E_{A_1}\cap E_{A_2}$. \\\\
Right panels: Plots of the ellipses associated to a different choice of the vector $\e$ at each step, from Example \ref{2x2_example}, which results in different upper bounds $B_1'$ and $B_2'$, with $B_2'$ minimal. Because the resulting minimal upper bound is incomparable to the previous, neither of the maximal ellipses in the bottom panels is contained in the other.\\\\
Note: All ellipses are centered at 0 and axis ticks are spaced 0.5 apart.}
\label{fig:2x2_example_ellipses}
\end{figure}
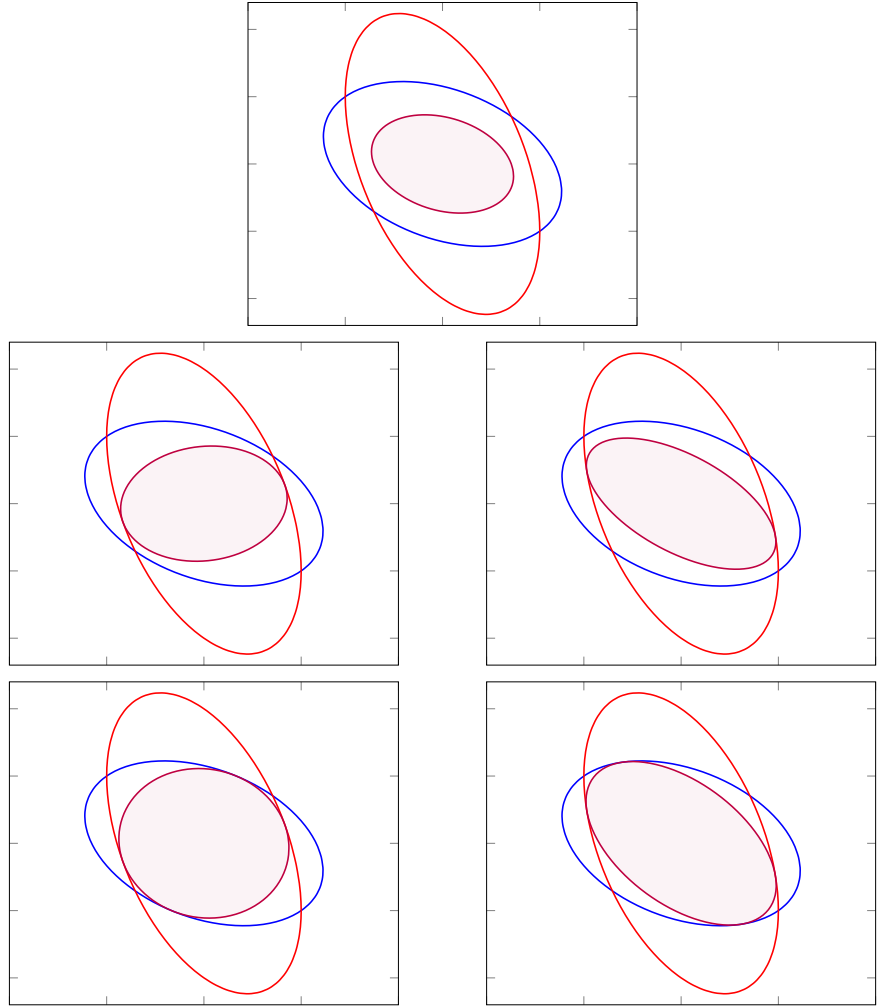

The next example provides an illustration of the algorithm with $3 \times 3$ matrices.

\begin{example}\label{3x3_example}
Consider the $3\times 3$ Hermitian matrices
\[
A_1 = 
\begin{bmatrix}
    \ph 2 & -1 & 0 \\
    -1 & \ph 2 & 0 \\
    \ph 0 & \ph 0 & 2
\end{bmatrix},\quad
A_2 = 
\begin{bmatrix}
    \ph 2 & -1 & 3 \\
    -1 &  \ph 2 & 0 \\
    \ph 3 & \ph 0 & 2
\end{bmatrix}
\]
and the initial upper bound
\[
B_0 =
\begin{bmatrix}
    \ph 4 & -1 & 3/2 \\
    -1 & \ph 3 & 0 \\
    \ph 3/2 & \ph 0 & 4
\end{bmatrix}.
\]
Algorithm \ref{main_algorithm} produces a minimal upper bound for $\{A_1,A_2\}$ in three steps. Successively using the vectors
\begin{align*}
\e_1 &=
[1 \; 1 \; 1]^\top,\\
\e_2 &=
[-5 \;4 \; 3]^\top, \\
\e_3 &=
[-9 \; -16\; 1]^\top
\end{align*}
in Step 3 of Algorithm \ref{main_algorithm} produces the minimal upper bound
\[
B_3 =
\frac{1}{2}
\begin{bmatrix}
    \ph 7 & -2 & 3 \\
    -2 &  \ph 4 & 0 \\
    \ph 3 & \ph 0 & 7
\end{bmatrix}
\]
for $\{A_1,A_2\}$.
\end{example}

\FloatBarrier

\section{Conclusion} \label{sec:conclusion}

In this paper, we studied the problem of computing a  minimal upper bound for a non-empty finite set of Hermitian matrices under the Loewner order. The main contribution of this paper is the development of an
an iterative method (Algorithm \ref{main_algorithm}) that exactly computes a minimal upper bound for any finite set of $n \times n$ Hermitian
matrices $\{A_1, \dots, A_k\}$ starting from any given upper bound $B_0$. It is proved that the algorithm
terminates in at most $n$ steps (Theorem \ref{thm:algorithm_terminates}), and more
precisely in exactly $m = n - \dim(E_0)$ steps, where $E_0$ denotes the
sum of the null spaces of the initial differences $B_0 - A_i, i \in \{1, \dots, k\}$. We also gave a self-contained proof of Stott's algebraic characterization of
minimality in the Loewner order, which
serves as the termination criterion
of the algorithm. The algorithm is illustrated on  examples, and a  Python implementation is available on GitHub (\href{https://github.com/patrickp78-byte/Minimal-Upper-Bound-Algorithm-in-Loewner-Ordering}{GitHub Repository}).

Many future research directions  have emerged from this work. It remains unknown whether every minimal upper bound for a given non-empty finite set of Hermitian matrices can be obtained by some sequence of direction vector choices $\mathbf{e}$ in
Algorithm \ref{main_algorithm}. Another direction to explore is the usefulness of the algorithm in semidefinite programming \cite{vandenberghe1996semidefinite}. It is unclear  whether minimizing a convex objective, such as the
trace or the log-determinant, over the set of upper bounds produces a
minimal upper bound in the Loewner sense, or yields a distinguished
minimal upper bound with additional optimality properties.  Semidefinite programming duality theory may also provide new certificates of minimality. The algorithm developed in this paper might also be valuable to generate a high-quality feasible starting matrix before using a semidefinite programming solver. Numerical experiments may be able to identify  situations leading to significant performance gain.

The theoretical results in this paper also lead naturally to some infinitary questions. For example, can the algorithm be modified to converge to a minimal upper bound for an infinite set of Hermitian matrices? For example, it is possible to prove the existence of minimal upper bounds for any compact subset of $M_n(\F)^\sa$. Another natural direction is to investigate minimality conditions such as Theorem \ref{minimal condition} in the setting of bounded operators on an infinite dimensional Hilbert space.

The connection
to the \emph{Loewner-John ellipsoid} \cite{henk2012lowner} might also be worth exploring. While Algorithm~\ref{main_algorithm} produces
an enclosed ellipsoid in at most $n$ steps, it remains open whether a
judicious choice of vector $\mathbf{e}$ at each iteration can recover the
maximum-volume enclosed ellipsoid (the inner Loewner-John ellipsoid) exactly.  Exploring
applications in control theory and quantum information where the Loewner
order governs comparisons of Lyapunov functions and density matrices, are other possible avenues.  Finally, a rigorous  numerical analysis of the algorithm is an obvious future research direction.

\medskip

\small \noindent \textbf{Acknowledgments.} 
The third and fourth authors would like to give a special thanks and gratitude for our supervisors, the first and second authors, for giving us the opportunity to work on this project. In addition, we also thank them for their patience, kindness, mentorship, and funding of this research to obtain the results written in this paper.

\bibliographystyle{plain}
\bibliography{sources}

@book{horn2012matrix,
  title={Matrix analysis},
  author={Horn, R.  and Johnson, C.},
  year={2012},
  publisher={Cambridge university press}
}

@phdthesis{stott2017minimal,
  title={Minimal upper bounds in the {L}{\"o}wner order and application to invariant computation for switched systems},
  author={Stott, N.},
  year={2017},
  school={Universit{\'e} Paris Saclay (COmUE)}
}

@article{kadison1951,
  title={Order properties of bounded self-adjoint operators},
  author={Kadison, R.},
  journal={Proceedings of the American Mathematical Society},
  volume={2},
  number={3},
  pages={505--510},
  year={1951},
  publisher={JSTOR}
}

@book{bhatia2013matrix,
  title={Matrix analysis},
  author={Bhatia, R.},
  volume={169},
  year={2013},
  publisher={Springer Science \& Business Media}
}

@book{bhimasankaram2010matrix,
  title={Matrix partial orders, shorted operators and applications},
  author={Bhimasankaram, P. and Malik, S. and Mitra, S.},
  volume={10},
  year={2010},
  publisher={World Scientific}
}

@article{Loewner1934,
  author  = {Loewner, C.},
  title   = {{\"U}ber monotone Matrixfunktionen},
  journal = {Mathematische Zeitschrift},
  volume  = {38},
  number  = {1},
  pages   = {177--216},
  year    = {1934},
  doi     = {10.1007/BF01170633}
}

@article{vandenberghe1996semidefinite,
  title={Semidefinite programming},
  author={Vandenberghe, L. and Boyd, S.},
  journal={SIAM review},
  volume={38},
  number={1},
  pages={49--95},
  year={1996},
  publisher={SIAM}
}

@book{GolubVanLoan2013,
  author    = {Golub, G. and Van Loan, C.},
  title     = {Matrix Computations},
  edition   = {4th},
  publisher = {Johns Hopkins University Press},
  address   = {Baltimore, MD},
  year      = {2013}
}

@article{HesteneStiefel1952,
  author  = {Hestenes, M. and Stiefel, E.},
  title   = {Methods of Conjugate Gradients for Solving Linear Systems},
  journal = {Journal of Research of the National Bureau of Standards},
  volume  = {49},
  number  = {6},
  pages   = {409--435},
  year    = {1952},
  doi     = {10.6028/jres.049.044}
}

@article{AshbyManteuffelSaylor1989,
  author  = {Ashby, S. and Manteuffel, T. and Saylor, P.},
  title   = {Adaptive Polynomial Preconditioning for {H}ermitian
             Indefinite Linear Systems},
  journal = {BIT Numerical Mathematics},
  volume  = {29},
  number  = {4},
  pages   = {583--609},
  year    = {1989},
  doi     = {10.1007/BF01932736}
}

@book{Anderson2003,
  author    = {Anderson, T.},
  title     = {An Introduction to Multivariate Statistical Analysis},
  edition   = {3rd},
  publisher = {Wiley-Interscience},
  address   = {Hoboken, NJ},
  year      = {2003}
}

@incollection{henk2012lowner,
  title={{L}{\"o}wner--{J}ohn ellipsoids},
  author={Henk, M.},
  booktitle={Optimization Stories},
  pages={95--106},
  year={2012},
  publisher={European Mathematical Society-EMS-Publishing House GmbH}
}

@article{BAKSALARY1989,
title = {Some properties of matrix partial orderings},
journal = {Linear Algebra and its Applications},
volume = {119},
pages = {57-85},
year = {1989},
issn = {0024-3795},
doi = {https://doi.org/10.1016/0024-3795(89)90069-4},
url = {https://www.sciencedirect.com/science/article/pii/0024379589900694},
author = {Baksalary, J. and Pukelsheim, F. and Styan, G.}
}

@article{BAKSALARY1990,
title = {Characterizations of the best linear unbiased estimator in the general {G}auss-{M}arkov model with the use of matrix partial orderings},
journal = {Linear Algebra and its Applications},
volume = {127},
pages = {363-370},
year = {1990},
issn = {0024-3795},
doi = {https://doi.org/10.1016/0024-3795(90)90349-H},
url = {https://www.sciencedirect.com/science/article/pii/002437959090349H},
author = {Baksalary, J. and Puntanen, S.}
}

@article{BAKSALARY1990171,
title = {Löwner-ordering antitonicity of generalized inverse of Hermitian matrices},
journal = {Linear Algebra and its Applications},
volume = {127},
pages = {171-182},
year = {1990},
issn = {0024-3795},
doi = {https://doi.org/10.1016/0024-3795(90)90342-A},
url = {https://www.sciencedirect.com/science/article/pii/002437959090342A},
author = {Baksalary, J. and Nordström, K. and Styan, G.}
}

@article{PAN1993103,
title = {A perturbation analysis of the problem of downdating a {C}holesky factorization},
journal = {Linear Algebra and its Applications},
volume = {183},
pages = {103-115},
year = {1993},
issn = {0024-3795},
doi = {https://doi.org/10.1016/0024-3795(93)90426-O},
url = {https://www.sciencedirect.com/science/article/pii/002437959390426O},
author = {Pan, C.}
}

@article{stott2016maximal,
  author  = {Stott, N.},
  title   = {Maximal lower bounds in the {L\"o}wner order},
  journal = {Proceedings of the American Mathematical Society},
  volume  = {146},
  year    = {2018},
  doi     = {10.1090/proc/13785},
}

@book{simon2019loewner,
  title={Loewner's theorem on monotone matrix functions},
  author={Simon, B.},
  volume={10},
  year={2019},
  publisher={Springer}
}

@book{boyd2004convex,
  title={Convex optimization},
  author={Boyd, S. and Vandenberghe, L.},
  year={2004},
  publisher={Cambridge university press}
}

@article{sherman1951order,
  title={Order in operator algebras},
  author={Sherman, S.},
  journal={American Journal of Mathematics},
  volume={73},
  number={1},
  pages={227--232},
  year={1951},
  publisher={JSTOR}
}
\end{document}